\documentclass[italian,twoside,a4paper]{amsproc}
\usepackage{amsmath}
\usepackage{amsthm}
\usepackage[applemac]{inputenc}

\usepackage{amssymb}
\usepackage{graphicx}
\usepackage{enumitem}
\usepackage{amsfonts}
\usepackage{array}
\usepackage{graphicx}
\usepackage[vcentermath,enableskew]{youngtab}
\newcommand{\f}{\varphi}
\newcommand{\G}{\mathbb G}
\newcommand{\R}{\mathbb R}

\newcommand{\PP}{\mathbb P}
\newcommand{\Q}{\mathbb Q}
\newcommand{\cE}{\mathcal E}
\newcommand{\cV}{\mathcal V}
\newcommand{\cS}{\mathcal S}
\newcommand{\cQ}{\mathcal Q}
\newcommand{\cO}{\mathcal O}
\newcommand\ra{\rightarrow}

\newcommand\pss{\par\smallskip}
\newcommand\lra{\longrightarrow}
\newcommand{\shse}[3]{0 ~\ra ~#1~ \lra ~#2~ \lra ~#3~ \ra~ 0}
\DeclareMathOperator{\ed}{e.\!d.}
\DeclareMathOperator{\rk}{rk}
\DeclareMathOperator{\HH}{H\hspace{0.5pt}}
\DeclareMathOperator{\DA}{A}
\DeclareMathOperator{\gd}{g.\!d.}

\newcolumntype{C}{>{$}c<{$}}

\newtheorem{theorem}{Theorem}[section]
\newtheorem{definition}[theorem]{Definition}
\newtheorem{proposition}[theorem]{Proposition}
\newtheorem{corollary}[theorem]{Corollary}
\newtheorem{example}[theorem]{Example}

\title{Morphisms between Grassmannians}
\author{Angelo Naldi and Gianluca Occhetta}
\begin{document}
\maketitle

\begin{abstract} Denote by $\G(k,n)$ the Grassmannian of linear subspaces of dimension $k$ in $\PP^n$. We show that if $n>m$ then every morphism $\f: \G(k,n) \to \G(l,m)$ is constant.
\end{abstract}

\section{Introduction}

In \cite{Tango1}, H. Tango proved that there are no nonconstant morphisms from $\PP^n$ to the Grassmannian $\G(l,m)$ if $n >m$; this result was later used by E. Sato in \cite{Satou} to study uniform vector bundles on  $\PP^n$. In fact, given a uniform vector bundle $\cE$ of rank $r$ on the projective space, one can construct a morphism from the variety of lines passing through a point $p$ -- which is $\PP^{n-1}$ -- to an appropriate Grassmannian of linear subspaces of $\PP(\cE_p)$, whose dimension depends on the splitting type of $\cE$. If those maps are constant for every $p$, then one can prove that the bundle $\cE$ is decomposable as a sum of line bundles.

The study of uniform bundles on other Fano manifolds brought along the problem to prove the constancy of morphisms $M \to \G(l,m)$, where $M$ is the Variety of Minimal Rational Tangents (VMRT) of the chosen Fano manifold.  Tango's proof made use of the particular structure of the cohomology ring of $\PP^n$, so it was first extended to varieties whose cohomology ring, up to some degree, is isomorphic to 
the one of the projective space  (see \cite[Lemma 3.4]{MOS1}). 

In \cite[Theorem 4.2]{Pan} the author showed that, in order to obtain a Tango type result, a weaker property of the cohomology ring was needed, the so called {\em good divisibility} up to some degree.
 A further step -- see \cite[Section 4]{MOS6} -- was to replace the notion of good divisibility with the weaker notion of {\em effective good divisibility} -- see Section \ref{sec:div} for the definitions of good divisibility and effective good divisibility. 
 
 The notion of effective good divisibility
 was used in \cite[Proposition 4.7]{MOS6} to prove the constancy of morphisms from a variety with ``large'' effective good divisibility not only to Grassmannians, but also  to other rational homogeneous varieties with Picard number one. The aim of \cite{MOS6}
was to prove the diagonalizability of uniform flag bundles of small invariants on some rational homogeneous varieties. Therefore in that paper the effective good divisibility has been computed   for projective spaces, quadrics and Grassmannians of lines, since those varieties appear as (factors of) the VMRT of the rational homogeneous varieties considered. 

However, leaving aside the applications to uniform bundles, the problem of computing the effective good divisibility in order to obtain Tango type results seems worth to be considered also for other varieties.

In the present paper we address it for the Grassmannians $\G(k,n)$ of linear subspaces of dimension $k$ of $\PP^n$; this choice seems a natural generalization of Tango's setting, since the Grassmannians $\G(k,n)$ are rational homogenous manifolds obtained as a quotient of the same group as $\PP^n$; moreover, the cohomology ring of Grassmannians has a clear description in terms of Schubert cycles, which are effective. 
We will prove that $\G(k,n)$ has effective good divisibility up to degree $n$, and derive from this the main result of the paper.

\begin{theorem}\label{main} If $n>m$ then every morphism $\f: \G(k,n) \to \G(l,m)$ is constant.
\end{theorem}

A straightforward consequence of Theorem \ref{main} is the following result.\par\smallskip

\begin{corollary}\label{cor:main} \! If $n>m$ then every morphism from $\G(k,n)$ to a rational homogeneous variety  of type $\DA_m$ is constant.
\end{corollary}

\section{Divisibility}\label{sec:div}

In this section we will recall the notions of good divisibility and effective good divisibility.
Let $M$ be a smooth complex projective manifold. Set
$\HH^{j}(M):=\HH^{j}(M,\R)$.

\begin{definition}\label{def:gd}\cite[Cf. Definition 4.1]{Pan}
A variety $M$ has {\em good divisibility up to degree $s$} if, given  $x_i \in \HH^{2i}(M),$  $x_j \in \HH^{2j}(M)$ with $i+j \le s$ and $x_i  x_j=0$, we have $x_i = 0$ or $x_j = 0$.
The {\em good divisibility} of $M$, denoted by $\gd(M)$, is the maximum integer $s$ such that $M$  has  good divisibility up to degree $s$.
\end{definition}

The concept of good divisibility can be refined, by considering only effective classes, i.e.,
classes that can be written as a  real linear combination of classes of subvarieties  with non negative coefficients.

\begin{definition}\cite[Definition 4.2]{MOS6}
\label{def:ed}
A variety $M$  has {\em effective good divisibility up to degree $s$} if, given  effective $x_i \in \HH^{2i}(M),$  $x_j \in \HH^{2j}(M)$ with $i+j \le s$ and $x_i  x_j=0$, we have $x_i = 0$ or $x_j = 0$.
The {\em effective good divisibility} of $M$, denoted by $\ed(M)$ is the maximum integer $s$ such that $M$  has  effective good divisibility up to degree $s$.
\end{definition}

It is clear from the definition that $\gd(M) \le \ed(M) \le \dim(M)$.

\begin{example} In \cite[Section 4]{MOS6} the divisibilities are computed in the following cases:
\begin{center}
\begin{tabular}{C||C|C}
M & \gd(M) & \ed(M)\\
\hline
\PP^n & n & n\\
\hline
\Q^{2n-1} & 2n-1 & 2n-1\\
\hline
\Q^{2n}  & n & 2n-1 \\
\hline
\G(1,n) & n-1 & n\\
\hline
\end{tabular}
\end{center}
\end{example}\medskip

The last two examples show that, in general $\gd(M) < \ed(M)$. 

The notion of effective good divisibility has been used in \cite[Section 4]{MOS6} to study morphisms to rational homogeneous varieties  of classical type of Picard number one. The following 
 is a special case of \cite[Proposition 4.7]{MOS6}, that we report here for the reader's convenience, since it is a key step in the proof of Theorem \ref{main}.

\begin{proposition}\label{thm2} Let $M$ be a smooth complex projective variety. If $ \ed(M) >m$ then any morphism $\f:M \to \G(l,m)$
is constant.
\end{proposition}

\begin{proof}
Consider the exact sequence of vector bundles on $\G(l,m)$:
\begin{equation}\shse{\mathcal{S}^\vee}{\cO^{\oplus m+1}}{\cQ}\label{eq:exact}\end{equation}
where $\cS$, of rank $m-l$, and $\cQ$, of rank $l+1$, are the universal subbundle and the universal quotient bundle. 
~Set 
\begin{align*}
\lambda_i&= \f^*c_i(\cQ)  & \text{for} \quad  & i=0, \dots, \rk \cQ \\
\mu_j&= \f^*c_j(\cS^\vee) & \text{for} \quad  & j=0, \dots, \rk \cS^\vee
\end{align*} 
The Chern classes of the nef bundles $\cQ$ and $\cS$ are effective and non zero, hence  $\lambda_i \not = 0$ is effective for every $i=0, \dots, \rk \cQ$ and $\mu_j \not=0$ is effective or antieffective for every $ j=0, \dots, \rk \cS^\vee$.
Set 
\[\lambda(t)  =\sum_i\lambda_it^i \quad \mu(t)=\sum_i\mu_it^i\]
by the exact sequence (\ref{eq:exact}) we have
$$\lambda(t)\mu(t)=1$$
Let $i_0$ and $j_0$ be the maximum indexes for which $\lambda_{i_0} \not = 0$ and  $\mu_{j_0} \not = 0$. 
If $ i_0+ j_0 \not =0$, then
$$\lambda_{  i_0}\mu_{j_0}\,=0,$$
but the assumption implies that $i_0+ j_0 \le  \ed(M)$, 
so we must have $i_0=j_0=0$.  In particular 
$$\lambda_1 =\f^*\det \cQ = \f^*\cO_{\G(l,m)}(1) = 0$$
 hence $\f$ is constant.
\end{proof}

\section{Schubert calculus}

We will now recall some basic facts about Schubert calculus in $\G(k,n)$. We refer to \cite{3264} and \cite{Fult} for a complete account on the subject.

Let us identify $\G(k,n)$ with the Grassmannian $G(k+1,V)$ of vector subspaces of dimension $k+1$ in a vector space $V$ of dimension $n+1$; consider a complete flag $\cV$ of vectors subspaces of $V$:
\[0 \subsetneq V_1 \subsetneq V_2 \subsetneq \dots \subsetneq V_n \subsetneq V\]
Given a sequence of integers $a=(a_1, \dots, a_{k+1})$ such that
\[n-k \ge a_1 \ge a_2 \ge \dots \ge a_{k+1} \ge 0\]
the Schubert variety $\Sigma_a(\cV)$ is defined as
\begin{multline*}
\Sigma_a(\cV) = \{W \in G(k+1,n+1)\ |  \dim(V_{n-k+i-a_i} \cap W) \ge i\ \text{for all}\ i \}
\end{multline*}
The codimension of $\Sigma_a(\cV)$ is $|a|:= \sum a_i$, the class $[\Sigma_a(\cV)] \in \HH^{2|a|}(\G(k,n))$ does not depend on the choice of $\cV$, will be denoted by $\sigma_a$ and called a Schubert cycle. The Schubert cycles form a basis for the cohomology of $\G(k,n)$.
 
The  cycle $\sigma_a$ may be represented by the corresponding Young diagram, which is a collection of left-justified rows of boxes in which the $i-th$ row has length $a_i$.
The product of two Schubert cycles can be computed using the the Littlewood-Richardson rule, as follows:
\begin{equation}\label{LR}
\sigma_a \cdot \sigma_b = \sum_{c} \delta^c_{a,b} \sigma_c
\end{equation}
where $|c|=|a|+|b|$, the Young diagram of $c$ contains the Young diagram of $a$ and
$\delta^c_{a,b}$ is equal to the number of Littlewood–Richardson tableaux of skew shape $c/ a$  and of weight $b$, i.e., to the number of ways in which on can fill the boxes of $c/a$ with integers $1, \dots, k+1$ in such a way that:
\begin{itemize}[topsep=5pt]
\item the integer $i$ appears $b_i$ times;
\item the integers are non decreasing along the rows (from left to right);
\item the integers are  strictly increasing along the columns (from up to down);
\item the word obtained concatenating the reversed rows is a lattice word, that is,  in every initial part of the sequence any number $i$ occurs at least as often as the number $i+1$.
\end{itemize}

\section{Proofs}

In this section we will prove Theorem \ref{main} and Corollary \ref{cor:main}.
In view of Proposition \ref{thm2} to prove Theorem \ref{main} it is enough to show that
 the effective good divisibility of $\G(k,n)$ is $n$. We start with a special case. 

\begin{proposition}\label{prop:hell}
Let $\sigma_a$ and $\sigma_b$ be two Schubert cycles in $\G(k,n)$, such that $|a|+|b| \le n$. Then $\sigma_a \cdot \sigma_b >0$.
\end{proposition}

\begin{proof}
Via the duality of $\PP^n$ we have isomorphisms $\G(k,n) \simeq \G(n-k-1,n)$, so we may assume that $2k \le n- 1$. In particular 
\begin{equation}\label{eq:1}
n-k > n/2.
\end{equation}
hence
\begin{equation}\label{eq:2}
|a|+|b| \le n < 2(n-k)
\end{equation}
By formula (\ref{LR}), since all the $\delta^c_{a,b}$ are nonnegative and the Schubert cycles are effective, it is enough to prove that  $\delta^c_{a,b}>0$ for some $c$.

We may assume that $a_1 + b_1 > n-k$, otherwise
taking $c=a+b$ and numbering the boxes in $c/a$ with the number of the row to which they belong, we find a Littlewood–Richardson tableaux of skew shape $c/a$  and of weight $b$.
 By  formula (\ref{eq:2})
we then have 
\begin{equation*}
a_i+b_j < n-k\ \text{for every} \ i,j \ge 2
\end{equation*}
In particular
\begin{equation}\label{eq:3}
b_j < n-k - a_2\ \text{for every} \ j \ge 2
\end{equation}
and
\[
\sum_{i=2}^{k+1}(a_i+b_i) < k
\]
Thus  $a_{k+1}=b_{k+1}=0$.


Set  $\bar a=(n-k-a_1, a_1-a_2, \dots, a_{k}-a_{k+1})$, and consider the skew shape $\bar a/a$, which contains $n-k$ boxes, no two of which on the same row. 

Mark $b_1$ boxes of $\bar a/a$ with the integer $1$, starting from the first row and moving from left to right. Denote by $\beta^1_j$ the number of boxes in the $j$-th row that have been marked with $1$'s and by $\beta^1$  the skew Tableaux consisting of the marked boxes in $\bar a /a$.

Set $a_1=(a_1+\beta^1_1, \dots, a_{k+1}+\beta^1_{k+1})$, 
and consider the sequence
\[\bar a^1= a^1 + (0, \beta_1^1, \beta_2^1,0, \dots, 0)\]
and the skew shape $\bar a^1/a^1$, which contains $\beta^1_1+\beta_2^1 $ boxes, no two of which are on the same row, since $a_3+\beta^1_3 \le a_2$.
Mark $b_2$ boxes of $\bar a^1/a^1$ with the integer $2$, starting from the first non empty row and moving from left to right. This is possible since $\beta^1_1+\beta_2^1 = \min\{b_1,n-k-a_2\}$, so, by (\ref{eq:3}), we have $b_2 \le \beta^1_1+\beta_2^1$.
Denote by $\beta^2_j$ the number of boxes in the $j$-th row that have been marked with $2$'s and by $\beta^2$  the skew Tableaux consisting of the marked boxes in $\bar a^1 /a^1$.

Repeat the procedure, starting at the step $i$ from the sequence 
\[\bar a^{i-1}= a^{i-1} + (0, \dots, 0, \beta_1^{i}, \beta_2^{i},0, \dots, 0)\]
where the two nonzero entries are the ones in positions $i,i+1$.
Consider the skew shape $\bar a^{i-1}/a^{i-1}$, which contains $\beta^{i-1}_1+\beta_2^{i-1} $ boxes, no two of which are on the same row, since $a_{i+1}+\beta^1_{i+1} \le a_i$.\\
Mark $b_i$ boxes of this skew shape with the integer $i$, starting from the first non empty row and moving from left to right. This is possible since $\beta^{i-1}_1+\beta_2^{i-1}  = b_{i-1}$, so, $b_i \le \beta^{i-1}_1+\beta_2^{i-1}$.  
Denote by $\beta^i_j$ the number of boxes in the $j$-th row that have been marked with $i$'s and by $\beta^i$  the skew Tableaux consisting of the marked boxes in $\bar a^{i-1} /a^{i-1}$.  Set $a^i=a^{i-1}+\beta^i$.

Set $c=a^k$; we claim that $\delta^c_{a,b}>0$.
The entries of the sequence $c$ are
\begin{align*}
c_1&=a_1+\beta^1_1 = n-k\\
c_2 &= a_2+\beta^1_2+\beta^2_1\\
\dots  &  \dots\\
c_j &= a_j+\beta^1_j+\beta^{j-1}_2 + \beta^j_1\\
\dots  &  \dots\\
c_{k+1} &=\beta^1_{k+1}+\beta^{k}_2 
\end{align*}
By construction $\beta^1_j+a_j \le a_{j-1}$ and $\beta_j^{i} \le \beta_j^{i-1}$, hence the sequence is not increasing. 

The Young diagram of the skew shape $c/a$ is the diagram of the Young tableaux $\beta^1+ \dots +\beta^k$; in this Tableaux by construction the rows are non decreasing and the columns are strictly increasing.
The fact that the word obtained concatenating the reversed rows is a lattice word follows from the inequalities  $\beta_j^{i} \le \beta_j^{i-1}$, which say that the number of $i-1$ in row $j+i-1$ is bigger than or equal to the number of
$i$ in row $j+i$.
\end{proof}

\begin{example}
The following is an example with $n=12, k=5$, to illustrate the case $a_1+b_1 \le n-k$ of the proof of Proposition \ref{prop:hell}\\
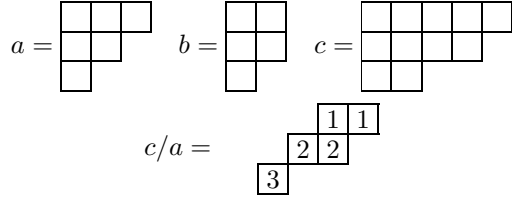
\[
a= \yng(3,2,1) \quad  b= \yng(2,2,1)\quad c=\yng(5,4,2)\]
\[
c / a = \,\, \young(:::11,::22,:3)
\]
\end{example} \pss
 
\begin{example}
The following is an example with $n=21, k=10$, to illustrate the case $a_1+b_1 > n-k$ of the proof of Proposition \ref{prop:hell}\\
\[
a= \yng(8,3,1) \quad  b= \yng(4,4,1)\]
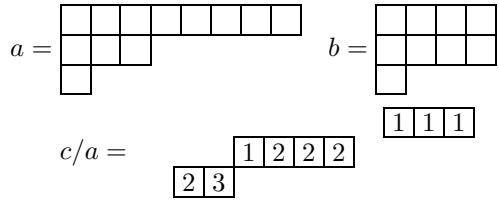
\[
\quad c/a = \,\, \young(::::::::111,:::1222,:23)
\]
\end{example}  \pss
 
We are now ready to compute the effective good divisibility of $\G(k,n)$. \pss
 
\begin{theorem}\label{cor:edg} The effective good divisibility of $\G(k,n)$ is $n$.
\end{theorem} 

\begin{proof}
By \cite[Corollary of Theorem 1]{FMSS}, as explained in \cite[Section 3]{Kopper}, 
the cones of effective classes of a fixed codimension $i$ in $\G(k,n)$ are polyhedral cones generated
by the Schubert classes of the same codimension.
Therefore, if  $\Gamma_i \in \HH^{2i}(\G(k,n))$, $\Delta_{j} \in  \HH^{2j}(\G(k,n))$ are two effective nonzero classes we may write:
$$
\Gamma_i=\sum_{|a|=i} \lambda_a \sigma_a,\quad \Delta_{j}=\sum_{|b|=j}\mu_b\sigma_{b},
$$
where  $\lambda_a, \mu_b \ge 0$.

Assume that $i+j \le n$; we want to show that $\Gamma_i \Delta_j\not=0$; since
every intersection $\sigma_{a} \sigma_{b}$ is a combination  of Schubert cycles, with nonnegative coefficients, due to the Littlewood--Richardson rule, it is enough to prove the statement for $\Gamma_i = \sigma_{a}$ and $\Delta_j=\sigma_{b}$, and this has been done in Proposition \ref{prop:hell}. Therefore $\ed(\G(k,n))\ge n$.

To show that $\ed(\G(k,n)) = n$ it is enough to consider $a=(1,1, \dots, 1)$ and $b=(k+1,0, \dots, 0)$. For the corresponding cycles we have $\sigma_a \sigma_b=0$, since there is no way of adding $k+1$ boxes marked with $1$ to the diagram of $a$ without having two boxes in the same column.
\end{proof}

As observed at the beginning of the section, Theorem \ref{main} follows combining Theorem \ref{cor:edg} and Proposition \ref{thm2}. 
\begin{proof}[Proof of Corollary \ref{cor:main}]
A rational homogeneous variety $X$ of type $\DA_m$ is a variety of partial flagsof linear subspaces  of dimensions $(l_1, \dots, l_i)$ in the projective space $\PP^m$. 

The variety $X$ has Picard number $i$, and morphisms $\pi_j \colon X \to \G(l_j,m)$ for every $j=1, \dots, i$, which send a partial flag $\Lambda^{l_1} \subset \Lambda^{l_2} \subset \dots \subset \Lambda^{l_i}$ to the point of $\G(l_j,m)$ parametrizing $\Lambda^{l_j}$.

If $\varphi \colon \G(k,n) \to X$ is a morphism, then $\pi_j \circ \varphi$ is constant for every $j=1, \dots, i$
by Theorem \ref{main}, hence $\varphi$ is constant.  
\end{proof}

\bibliographystyle{plain}

\end{document}